\documentclass[12pt,article,amstex]{amsart}

\usepackage{amsmath,amsthm}
\usepackage{amssymb}
\usepackage{enumerate}

\def\ds{\displaystyle}

\def\ds{\displaystyle}

\newcommand{\R}{{\Bbb R}}

\newcommand{\M}{{\Bbb M}}
\newcommand{\A}{{\Bbb A}}

\newcommand{\loc}{{\rm loc}}
\newcommand{\sign}{{\rm sign\,}}

\newcommand{\grad}{{\rm grad\,}}

\newcommand{\dist}{{\rm dist\,}}
\newcommand{\mO}{{\mathcal O}}

\def\Xint#1{\mathchoice 
{\XXint\displaystyle\textstyle{#1}}%
{\XXint\textstyle\scriptstyle{#1}}%
{\XXint\scriptstyle\scriptscriptstyle{#1}}%
{\XXint\scriptscriptstyle\scriptscriptstyle{#1}}%
\!\int} 
\def\XXint#1#2#3{{\setbox0=\hbox{$#1{#2#3}{\int}$ } 
\vcenter{\hbox{$#2#3$ }}\kern-.6\wd0}}

\def\tl{\widetilde}
\def\ol{\overline}

\def\B{{\mathcal B}}

\def\F{{\Bbb  F}}

\def\NN{{\mathcal N}}

\def\bA{{\mathbf A}}

\def\div{{\rm div\,}}

\def\loc{{\rm loc}}
\def\ba{{\mathbf a}}
\def\bu{{\mathbf u}}
\def\bz{{\mathbf z}}

\def\bb{{\mathbf b}}
\def\diam{{\rm diam\,}}

\newtheorem{thm}{Theorem}[section]  
\newtheorem{lem}[thm]{Lemma}	       
\newtheorem{crlr}[thm]{Corollary}      
\newtheorem{defin}[thm]{Definition}    
\newtheorem{rem}[thm]{Remark}	       

\numberwithin{equation}{section}
 
\begin{document}

\baselineskip=14pt

\title[Quasilinear systems in Morrey spaces]{Precise Morrey regularity  of the  weak   solutions to  a kind of quasilinear  systems with discontinuous  data}

\author[L. Fattorusso]{Luisa Fattorusso}
\address{       Department of Information Engineering,
 Infrastructure and Sustainable Energy,
Mediterranea University,
 of Reggio Calabria, Italy\\
 e-mail:  luisa.fattorusso@unirc.it}

\author[L. Softova]{Lubomira Softova}
\address{Department of Mathematics
University of Salerno, Italy\\
 e-mail:  lsoftova@unisa.it}

\maketitle
\renewcommand{\thefootnote}{}



\renewcommand{\thefootnote}{\arabic{footnote}}
\setcounter{footnote}{0}

\begin{abstract}
We consider the Dirichlet problem for a class of quasilinear elliptic systems   in domain  with irregular boundary.
 The principal part satisfies componentwise coercivity condition   and  the nonlinear   terms  are Carath\'eodory maps having Morrey regularity in $x$ and verifying controlled growth conditions with respect to the other variables.
 We have obtained boundedness  of the    weak  solution  to the problem that  permits to apply an iteration procedure in order to find  optimal   Morrey regularity of its   gradient.
\end{abstract}

\section{Introduction}

We are interested in  the regularity properties of a kind of quasilinear elliptic  operators with discontinuous data acting in a bounded domain $\Omega,$   with irregular  boundary $\partial\Omega.$  Precisely, we consider the following Dirichlet problem 
\begin{equation}\label{DP}
\begin{cases}
\div\big(\bA(x) D\bu  +\ba(x,\bu) \big)= \bb(x,\bu,D\bu) &  x\in\Omega \\
\bu(x)=0\quad  & x\in  \partial\Omega\,.
\end{cases}
\end{equation}
Here  $\Omega\subset\R^n,$ $n\geq 2$ is a bounded  {\it Reifenberg-flat  domain}, the coefficients  matrix
$\bA=\{A_{ij}^{\alpha\beta}(x)\}^{\alpha,\beta\leq n}_{i,j\leq N}$ is   essentially bounded  in $\Omega$  and the non linear terms 
$$
\ba(x,\bu)=\{a^\alpha_i(x,\bu)\}^{\alpha\leq n}_{i\leq N} \quad \text{ and  } \quad 
\bb(x,\bu,\bz)=\{b_i(x,\bu,\bz)\}_{i\leq N}
$$
 are  Carath\'eodory maps, i.e., they are measurable in $x\in \Omega $ for all    $\bu\in \R^N,$ $\bz \in \M^{N\times n}$ and continuous in $(\bu,\bz)$ for almost all $x\in \Omega.$   Since we are going to study the  weak solutions of \eqref{DP} we need  to impose  {\it controlled growth conditions} on the nonlinear  terms  in order to ensure  convergence of the integrals in the definition  \eqref{weak}. 
For this aim we suppose that (cf.  \cite{G,Sf1})
\begin{align*}
a^\alpha_i(x,\bu)&= \mO(\varphi_1(x) +|\bu|^{\frac{n}{n-2}}), \\  
b_i(x,\bu,\bz)&=\mO(\varphi_2(x)+|\bu|^{\frac{n+2}{n-2}}+|\bz|^{\frac{n+2}{n}})
\end{align*}
for $n>2.$  In the particular case $ n = 2,$ the powers of
$|\bu|$ could be arbitrary positive numbers, while the growth of  $|\bz|$ is subquadratic.

Our aim is to study the dependence of the solution from the regularity of the data  and to obtain 
Calder\'on-Zygmund type estimate  in an optimal Morrey space.

There are  various papers dealing with  the integrability and regularity properties  of  different kind of  quasilinear and nonlinear differential   operators. Namely, it is studied the question {\it  how the regularity of the data influences  on the regularity of the  solution.} In the scalar case  $N=1$  the celebrated result of De Giorgi and Nash  asserts that  {\it  the weak solution of linear elliptic and parabolic   equations with only $L^\infty$ coefficients is H\"older continuous} \cite{DeG}   

 Better integrability can be obtained also by the  result of Gehring \cite{Gh} relating to {\it  functions satisfying the inverse  H\"older inequality.}  Later Giaquinta and Modica \cite{GMo} noticed that certain power of the gradient of a function $u\in W^{1,p}$ satisfies locally the reverse H\"older inequality. Modifying Gehring's lemma  they  obtained  {\it  better integrability for the weak solutions of some quasilinear elliptic equations.} Their pioneer  works have been followed  by extensive research  dedicated to the regularity properties of various partial differential   operators  using the  Gehring-Giaquinta-Modica technique,   called also a  {\it  ''direct method"} (cf.  \cite{Ar,PS1,PS2} and the references therein.)   Recently,  {\it the method of $A$-harmonic approximation}  permits to study  the  regularity    without using  Gehring's lemma  (see for example \cite{AM}).

The theory for  linear divergence form   operators  defined in Reifenberg's domain   was developed firstly  in    \cite{BW1, BW}. 
In \cite{BP1, BPSh} the authors extend this theory to  quasilinear uniformly elliptic equations in the Sobolev-Morrey spaces. Making use of the  Adams  inequality  \cite{Ad} and the  Hartmann-Stampacchia maximum principal they obtain H\"older regularity of the solution while in \cite{BS1} it is obtained generalized H\"older regularity for regular  and nonregular nonlinear elliptic equations. 

As it concerns nonlinear nonvariational operators
 we can mention the results of Campanato \cite{C} relating  to basic  systems of the form $F(D^2u)=0 $ in  the  Morrey  spaces. 
Afterwards  Marino and  Maugeri in \cite{MM} have contributed  to 
this  theory  with their own research in the boundary regularity about the basic systems. 
Imposing differentiability of the operator $F $ they obtain, via immersion theorems, Morrey regularity of the second derivatives 
$D^2u\in L^{2,2-\frac{2}{q}}, q>2.$ These studies have been extended in   \cite{FT} to nonlinear 
   equations of a kind  $F(x,D^2u) $ without  any differentiability assumptions  on $F.$  It is obtained  global  Morrey regularity   via the {\em Korn trick} and the {\it near operators theory}  of Campanato. Moreover,   in the variational case it is  established  a Caccioppoli-type inequality for a second-order degenerate elliptic systems of $p$-Laplacian type \cite{FMT}.  Exploiting the   classical Campanato's approach and the {\it hole-filling technique} due to Widman, it is proved a global regularity result for the gradient of $\bu$ in the  Morrey and   Lebesgue  spaces. 

 \smallskip

In the present work we consider quasilinear systems in divergence form with a principal part satisfying {\it componentwise coercivity condition.} This condition permits to apply the result of \cite{PSf3,Sf1}  that gives $L^\infty$ estimate  of  the weak  solution.  In addition the  {\it controlled growth conditions} imposed  on the nonlinear  terms allow to apply the integrability result from \cite{Sf2}. Making use of step-by-step technique we show optimal Morrey regularity of the gradient depending explicitly on the regularity of the data.

\smallskip

 In what follows  we use the  standard  notation:
\begin{itemize}
\item $\Omega$ is a bounded domain in $\R^n,$ with a Lebesgue measure $|\Omega|$ and boundary $\partial\Omega;$
\item
 $B_\rho(x)\subset \R^n$  is a ball,  
$\Omega_\rho(x)=\Omega\cap B_\rho(x)$ with  $ \rho\in (0,\diam\Omega],$  $x\in \Omega;$
\item   
  $\M^{N\times n}$  is   the set of  $N\times n$-matrices.
\item    $\bu= (u^1,\ldots,u^N):\Omega\to \R^N, $ \quad   $ \ds   D_\alpha u^j=\partial u^j/\partial x_\alpha, 
$\\
$\ds |\bu|^2=\sum_{j\leq N} |u^j|^2, $ 
  \quad
$
  D\bu=\{D_\alpha u^j\}^{\alpha\leq n}_{j\leq N}\in \M^{N\times n}, $\\
$  
\ds  |D\bu|^2=\sum_{\underset{j\leq N}{\alpha \leq n}}|D_\alpha u^j|^2 ;
$
\item  For $\bu\in L^p(\Omega;\R^N)$ we write $\|\bu\|_{p,\Omega}$ instead of
 $\|\bu\|_{ L^p(\Omega;\R^N)}$
\item 
The spaces   $W^{1,p}(\Omega;\R^N)$ and  $W_0^{1,p}(\Omega;\R^N)$ are the classical Sobolev spaces as they are defined in \cite{GT}.
\end{itemize} 

 Through all the paper the standard summation convention on repeated upper and lower indexes is adopted. 
The letter $C$ is used for various positive  constants and may change from one occurrence to another.

\section{Definitions and auxiliary  result}

In \cite{R}  Reifenberg  introduced a class of domains with rough  boundary  that  can be approximated locally  by hyperplanes. 
\begin{defin}\label{reifenberg}\rm
The domain $\Omega$ is $(\delta,R)$  Reifenberg-flat if there exist positive constants $R$ and  $\delta<1 $ such that for each $x\in \partial \Omega$
and each $\rho\in(0,R)$ there is a local coordinate system $\{y_1,\ldots,y_n\}$ with the property
\begin{equation}\label{eq2}
\B_\rho(x)\cap \{y_n>\delta\rho \}\subset\Omega_\rho(x)\subset\B_\rho(x)\cap\{y_n>-\delta\rho \}.
\end{equation}
\end{defin}
Reifenberg  arrived at this concept of flatness in his studies on the  Plateau   problem  in higher dimensions and he proved that such a domain is locally a topological disc when $\delta$  is small enough, say $\delta<1/8.$   It is easy to see that 
a  $C^1$-domain is  a  Reifenberg flat  with $\delta\to 0 $ as $R\to 0.$  A domain with Lipschitz boundary with a Lipschitz  constant less than $\delta$ also verifies the condition \eqref{eq2} if  $\delta$  is small enough, say $\delta<1/8$,  (see \cite[Lemma~5.1]{BW}).      But the class of Reifenberg's  domains is much  more wider  and contains domains with fractal boundaries. For instance, consider
a self-similar snowflake  $S_\beta.$  It is  a flat version of the Koch snowflake $S_{\pi/3}$  but with   angle of the spike $\beta$ such that
$\sin\beta\in(0,1/8).$   
This kind of flatness exhibits  minimal geometrical conditions necessary for some natural properties  from the  analysis and potential theory to hold.  For more detailed overview  of  these  domains
we refer the reader to  \cite{T} (see also \cite{BW1, PS1} and the references therein).

In addition  \eqref{eq2}  implies the $(A)$-property  (cf. \cite{G,PS2}).  
Precisely,  there exists a positive  constant  $A(\delta)<1/2$ such that 
\begin{equation}\tag{A}\label{A}
 A(\delta)|\B_\rho(x)| \leq |\Omega_\rho(x)|\leq (1- A(\delta))|\B_\rho(x)|
\end{equation}
for any fixed $x\in\partial\Omega,$ $\rho\in(0,R)$ and $\delta\in(0,1).$  This condition excludes that $\Omega $  may have sharp outward and inward cusps.
As consequence, the Reifenberg domain  is  $W^{1,p}$-extension domain, $1\leq p\leq \infty,$ hence 
the usual extension theorems, the  Sobolev and Sobolev-Poincar\'e inequalities are still valid in $\Omega$ up to the boundary.  
\begin{defin}\label{Morrey}
A real valued function $f\in L^p(\Omega)$ belongs  to the Morrey space $L^{p,\lambda}(\Omega)$ with $p\in[1,\infty),$ $\lambda\in(0,n),$   if  
$$
 \|f\|_{p,\lambda;\Omega}= \left(\sup_{\B_\rho(x) }\frac{1}{\rho^\lambda}\int_{\Omega_\rho(x)} |f(y)|^p\, dy \right)^{1/p} <\infty
$$
where $\B_\rho(x)$ ranges in the set of all balls with radius $\rho \in(0,\diam\Omega]$  and  $x\in \Omega.$ 
\end{defin}

In  \cite{Mo} Morrey obtained local H\"older regularity of the solutions to second order elliptic equations. His new approach consisted  in estimating the growth of the integral function  $g(\rho)=\int_{\B_\rho}|Du(y)|^p dy$ via a  power of the radius of the same ball, i.e.,  $C \rho^\lambda$ with $\lambda>0.$  Nevertheless that he did not talk about function spaces, his paper is considered as the starting point for the theory of the {\it Morrey spaces}   $L^{p,\lambda}. $

The family of the   $L^{p,\lambda} $   spaces is partially ordered.  (cf. \cite{Pi}).
\begin{lem}\label{lem1}
 For $1\leq r'\leq r''<\infty$ and $\sigma',\sigma''\in[0,n)$ the following embedding holds
\begin{equation}\label{MorEmbed}
L^{r''\sigma''}(\Omega)\hookrightarrow  L^{r',\sigma'}(\Omega)\qquad \text{ iff } \quad \frac{n-\sigma'}{r'}\geq \frac{n-\sigma''}{r''}\,.
\end{equation}
Furthermore, we have the continuous inclusion 
\begin{equation}\label{MorIncl}
L^{\frac{n r'}{n-\sigma'}}(\Omega) \hookrightarrow  L^{r',\sigma'}(\Omega)\,.
\end{equation}
\end{lem}

For $x\in \R^n,$   $I_\alpha$  is the {\it  Riesz potential operator\/} whose convolution kernel is  $|x|^{\alpha -n},$ $ 0<\alpha<n.$
  Suppose that  $f$ is extended as zero in $\R^n$ and consider its Riesz potential   $I_\alpha f(x)=\int_{\R^n}\frac{f(y)}{|x-y|^{n-\alpha}}dy.$ In \cite{Ad} Adams obtained the following   inequality. 
\begin{lem}\label{lem2}   Let
 $f\in L^{r,\sigma}(\R^n),$ then $I_\alpha: L^{r,\sigma} \to L^{ r_\sigma^*,\sigma}$ is continuous  and
\begin{equation}\label{eqRiesz}
\|I_\alpha f\|_{L^{r_\sigma^*,\sigma}(\R^n)}\leq C\|f\|_{L^{r,\sigma}(\R^n)}, \qquad
\end{equation}
where $C$ depends on $n,r,\sigma,|\Omega|, $   and  $r_\sigma^*$ is the {\it Sobolev-Morrey conjugate}
\begin{equation}\label{eqSM}
 r_\sigma^*=\begin{cases}
\frac{(n-\sigma)r}{n-\sigma-r} & \text{ if } r+\sigma<n\\
\text{arbitrary large number } & \text{ if } r+\sigma\geq n\,.
\end{cases}
\end{equation}
\end{lem}

The nonlinear terms 
    $\ba(x,\bu)$ and $\bb(x,\bu,\bz)$  satisfy  {\it controlled growth conditions}
\begin{align}\label{contr1}
 & |\ba(x,\bu) |\leq  \Lambda(\varphi_1(x)+|\bu|^{\frac{2^*}{2}}), \\
\nonumber
& \qquad\qquad   \varphi_1\in L^{p,\lambda}  (\Omega), \quad   p>2 , \      p+\lambda>n, \  \lambda\in[0,n),\\
\label{contr2}
&  |\bb(x,\bu,\bz)|\leq   \Lambda\big(\varphi_2(x) + |\bu|^{2^*-1}+|\bz|^{2\frac{(2^*-1)}{2^*}}  \big), \\
\nonumber
&\qquad \qquad  \varphi_2\in L^{q,\mu} (\Omega), \quad   q>\frac{2^*}{2^*-1}, \    2q+\mu>n, \  \mu\in[0,n)
\end{align}
with a positive constant $\Lambda. $ Here $2^*$ is te Sobolev conjugate of 2, i.e.
$2^*=\frac{2n}{n-2 }$  if $ n>2$ and it is arbitrary large number  if $n=2$  (cf. \cite{G, LU,Sf2,Sf1}).

A {\it weak solution} to  \eqref{DP} is  a function  
 $\bu\in W_0^{1,2}(\Omega;\R^N)$   satisfying
\begin{align}\label{weak}
\nonumber
\int_\Omega& A^{\alpha\beta}_{ij}(x)D_\beta u^j(x)D_\alpha\chi^i(x) dx +\int_\Omega 
a^\alpha_i(x,\bu(x))D_\alpha\chi^i(x)dx\\
& +\int_\Omega b_i(x,\bu(x),D \bu(x))\chi^i(x) dx=0,\quad j=1,\ldots,N
\end{align}
for all $\chi\in W_0^{1,2}(\Omega;\R^N)$ where  the convergence of the integrals is ensured by   \eqref{contr1}  and \eqref{contr2}.

\section{Main result}

The general theory of elliptic systems  does not ensure boundedness of the solution if we  impose  only   growth conditions  as  \eqref{contr1} and \eqref{contr2}   (see for example \cite{JS, LP}).    For this goal we need some additional structural restrictions on  the operator as  {\it componentwise coercivity} similar to that imposed in \cite{LP,PSf3,  Sf3,Sf1}. 
 
Suppose that  $\|\bA\|_{\infty,\Omega}\leq \Lambda_0$ and for each fixed  $i\in\{1,\ldots,N\}$ there exist positive constants  $\theta_i$ and  $ \gamma(\Lambda_0)$  such that  for  $|u^i|\geq \theta_i $  we have
\begin{equation}\label{coercivity}
\begin{cases}
\ds \gamma|\bz^i|^2 -\Lambda |\bu|^{2^*} - \Lambda \varphi_1(x)^2\leq \sum_{\alpha=1}^n\left(
A_{ij}^{\alpha\beta} (x) z^j_\beta+ a_i^\alpha(x,\bu) \right) z^i_\alpha\\
  \ds  b_i(x,\bu,\bz) \, \sign u^i(x)\geq -\Lambda\left( \varphi_2(x)+|\bu|^{2^*-1}+|\bz^i|^{2\frac{2^*-1}{2^*}}  \right)
\end{cases}
\end{equation}
for a.a.  $ x\in\Omega$  and for all $\bz\in \M^{N\times n}.$ The functions $\varphi_1$ and $\varphi_2$ are  as in \eqref{contr1} and \eqref{contr2}.
 \begin{thm}\label{th1}
Let $\bu\in W_0^{1,2}(\Omega;\R^N)$ be a weak solution of the problem \eqref{DP} under the conditions \eqref{eq2}, 
\eqref{contr1},   \eqref{contr2} and \eqref{coercivity}.
Then
\begin{equation}\label{qast}
\bu\in W_0^{1,r}\cap  L^\infty(\Omega;\R^N)\  \text{ with }\   r=\min\{p,q^*_\mu\}\,.
\end{equation}
Moreover  
\begin{equation}\label{gradient}
|D\bu|\in L^{r,\nu}(\Omega)\quad \text{ with }\quad \nu=\min\Big\{n+ \frac{r(\lambda-n)}{p}, n+ \frac{r(\mu-n)}{q^*_\mu}  \Big\}
\end{equation}
where $q^*_\mu$ is the Sobolev-Morrey conjugate of $q$ (see \eqref{eqSM}).
\end{thm}
\begin{rem}\em
If we take  bounded weak solution of \eqref{DP}, i.e., $\bu\in W_0^{1,r}\cap L^\infty(\Omega; \R^N) $ we can substitute  the coercivity condition \eqref{coercivity} with a uniform ellipticity condition. In this case we may suppose the principal coefficients to be discontinuous with small discontinuity controlled by their $BMO$ modulus. Precisely,  we suppose that  
\begin{align*}
\sup_{0<\rho\leq R}\sup_{y\in\Omega}&\  \Xint-_{\Omega_\rho(y)} |A^{\alpha\beta}_{ij}(x)-
\overline{A^{\alpha\beta}_{ij}}_{\Omega_\rho(y)}|^2\, dx\leq \delta^2,\\
&\overline{A^{\alpha\beta}_{ij}}_{\Omega_\rho(y)}=\Xint-_{\Omega_\rho(y)} {A^{\alpha\beta}_{ij}}(x)\, dx,
\end{align*}
where $\delta\in(0,1)$ is the same parameter as in \eqref{eq2}.  
The small $BMO$ successfully substitute the   $VMO$ in the study of PDEs with discontinuous coefficients,  harmonic analysis and integral operators studying,  geometric measure analysis and differential geometry (see \cite{BP1, BSf, BW1, GPVC, PS2, Sf1} and the references therein).
A higher integrability result for such kind of operators  can be found in \cite{DK,PS2,Sf2} for equations and  systems, respectively. 
\end{rem}
\begin{proof}
The essential boundedness of the solution follows by \cite{PSf3} (see also \cite{Sf3,Sf1}).  Precisely, there exists a constant depending on $n,$ $\Lambda,$ $p,$ $q,$ $\|\varphi_1\|_{L^p(\Omega)},$ $ \|\varphi_2\|_{L^q(\Omega)}$ and $\|D\bu\|_{L^2(\Omega)}$  such that 
\begin{equation}\label{boundedness}
\|\bu\|_{\infty, \Omega}\leq M\,.
\end{equation}
 Let  the solution and the functions $\varphi_1$ and $\varphi_2$ be extended as zero outside $\Omega.$ 
By the Definition~\ref{Morrey} we have that 
 $ \varphi_{1} \in L^{p}(\Omega) $ and  $ \varphi_{2} \in L^{q}(\Omega). $  In \cite{G} Giaquinta  show that there exists an exponent $\tl r>2$ such that $\bu\in W_\loc^{1,\tl r}(\Omega;\R^N).$ His  approach is based on the reverse H\"older inequality and a version of  Gehring's lemma. Since the Cacciopoli-type inequalities hold up to the boundary, this method can be carried out up to the boundary and it is done in \cite[Chapter 5]{G} for the Dirichlet problem in Lipschitz domain (see also \cite{Ar,C,DK,Sf2}).
In \cite{BW2}   the authors  have shown that an 
 inner neighborhood of $(\delta,R)$--Reifenberg flat domain is a Lipschitz domain with the  $(\delta,R)$-Reifenberg flat
property.  More precisely, we dispose with the following result.
\begin{lem}(\cite{BW2})
Let $\Omega$  be a $(\delta,R)$-Reifenberg flat domain for sufficiently small  $\delta>0.$ Then for any $0<\varepsilon<\frac{R}{5}$ the set $\Omega_\varepsilon=\{ x\in\Omega:\dist)x,\partial\Omega)>\varepsilon  \} $ 
 is a Lipschitz domain  with the  property \eqref{eq2}.
\end{lem}
This lemma  permits us to extend the results of  \cite[Chapter 5]{G} in   Reifenberg-flat domains.
Further, by    \cite{Sf2}     $|D\bu|$ belongs at least to $ L^{r_0}(\Omega)$ with $r_0=\min\{p,q^*\}>\frac{n}{n+2}.$

Let \fbox{$n> 2$} and     $\bu\in W_0^{1,{r_0}}(\Omega;\R^N)\cap L^\infty(\Omega;\R^N)$  be a solution to   \eqref{DP}.   Our {\it first step}  is to improve its integrability.  Fixing that solution in the nonlinear terms we get the linearized problem
\begin{equation}\label{eq3}
\begin{cases}
 D_\alpha\big(A^{\alpha\beta}_{ij}(x) D_\beta   u^{j} (x)\big) 
 \big)= f_i(x) - D_\alpha A_i^\alpha(x) &  x\in\Omega \\
\bu(x)=0\quad  &  x\in  \partial\Omega
\end{cases}
\end{equation}
where we have used the notion 
$$
f_i(x)= b_i(x,\bu,D\bu),\qquad A_i^\alpha(x)= a_i^\alpha(x,\bu),  
$$

By  \eqref{contr1}, \eqref{contr2} and \eqref{boundedness}  we get
\begin{equation}\label{est1}
|A_i^\alpha(x)| \leq \Lambda \left(\varphi_1(x)+ |\bu(x)|^{\frac{n}{n-2}}\right)
\end{equation}
that gives $A_i^\alpha(x)\in L^{p,\lambda}(\Omega) $ with $p>2$ and $p+\lambda>n.$  Analogously
\begin{equation}\label{est2}
|f_i(x)|\leq \Lambda\left( \varphi_2(x) +|\bu|^{\frac{n+2}{n-2}} +|D\bu|^{\frac{n+2}{n}}  \right)\,.
\end{equation}
Since $|D\bu|\in L^{r_0}(\Omega)$  we get  $|D\bu|^{\frac{n+2}{n}}\in L^{\frac{r_0n}{n+2}}(\Omega)$ that gives
 $f_i\in L^{q_1}(\Omega)$ where   $q_1=\min\{q,\frac{r_0 n}{n+2}\}.$

Let $\Gamma$ be the fundamental solution of the Laplace operator. Recall that the {\it  Newtonian potential\/}  of $f_i(x)$ is  given by
$$
\NN f_i(x)=\int_\Omega \Gamma(x-y) f_i(y)\, dy, \qquad \Delta \NN f_i(x)=f_i(x) \  \text{ for a.a. } x\in \Omega
$$
and by \cite[Theorem~9.9]{GT} we have that  $\NN f_i\in W^{2,q_1}(\Omega).$
Denote by 
$$
F_i^\alpha(x)=D_\alpha \NN f_i(x)=C(n)\int_\Omega \frac{(x-y)_\alpha f_i(y)}{|x-y|^n}\,dy   \quad  \text{ for a.a. } x\in \Omega
$$ 
and $\F_i=(F_i^1,\ldots,F_i^n)=\grad \NN f_i.$ Hence $\div \F_i=f_i$ and 
\begin{equation}\label{eq4}
\begin{cases}
 D_\alpha\big(A^{\alpha\beta}_{ij}(x) D_\beta   u^{j} (x)\big) 
 \big)=  D_\alpha ( F_i^\alpha(x)-A_i^\alpha(x)) &  x\in\Omega \\
\bu(x)=0\quad  & x\in \partial\Omega
\end{cases}
\end{equation}
By \eqref{est1} and \eqref{est2} we get
\begin{align}\label{FAest}
\nonumber
|F_i^\alpha(x)-A_i^\alpha(x)|  \leq &\  C(n,\Lambda)\int_\Omega \frac{ \varphi_2(y) +|\bu(y)|^{\frac{n+2}{n-2}} +|D\bu(y)|^{\frac{n+2}{n}}     }{|x-y|^{n-1}}\,dy\\
&\ + \Lambda \left(\varphi_1(x)+ |\bu(x)|^{\frac{n}{n-2}}\right)\\
\nonumber
 \leq &\  C\left(1+\varphi_1(x)+I_1\varphi_2(x)+ I_1 |D\bu(x)|^{\frac{n+2}{n}}  \right)
\end{align}
with a constant depending on  $n,\Lambda,$  and $ \|\bu\|_{\infty,\Omega}.$ 
By \eqref{eqRiesz} we get
\begin{align}\label{eq6}
&\|I_1\varphi_2\|_{L^{q^*_\mu,\mu}(\Omega)}\leq C\|\varphi_2\|_{L^{q,\mu}(\Omega)}\\
\label{eq6a}
&\| I_1|D\bu|^{\frac{n+2}{n}} \|_{L^{(\frac{r_0n}{n+2})^*}(\Omega)}\\
\nonumber
&\   \leq C  
\|\left| D\bu\right|^{\frac{n+2}{n}}  \|_{L^{\frac{r_0n}{n+2}}(\Omega)}\leq C\| D\bu  \|^{\frac{n+2}{n}}_{L^{r_0}(\Omega)}
\end{align}
where  $q^*_\mu$ is the Sobolev-Morrey conjugate of $q$ and 
$$
\left(\frac{r_0n}{n+2} \right)^*=\begin{cases}
\ds \frac{r_0n}{n+2-r_0} & \text{ if } r_0<n+2\,,\\
\ds \text{arbitrary large number } & \text{ if } r_0\geq n+2\,.
\end{cases}
$$
Hence $F_i^\alpha - A_i^\alpha \in L^{r_1}(\Omega)$ with $ r_1=\min\{p, q^*_\mu, (\frac{r_0n}{n+2})^* \}.$ If  
$ r_1=\min\{p, q^*_\mu \}$ then we have the assertion, otherwise $ r_1=(\frac{r_0n}{n+2})^* $ and we consider two cases:
\begin{enumerate} 
\item $r_0=p $
that leads to  $ p>  (\frac{pn}{n+2})^*$ which is impossible;
\item $r_0=q^*$ and  we consider  two subcases:
\begin{itemize}
\item[2a)] $q^*\geq n+2$ which means that $r_1$ is arbitrary large number and we arrive to contradiction with the assumption $r_1<\min\{p, q^*_\mu\};$

\item[2b)] $q^*<n+2$ hence $r_1=\frac{q^*n}{n+2-q^*}.$
\end{itemize}
\end{enumerate} 
 Applying \cite[Theorem~1.7]{BW} to the linearized  system \eqref{eq4}  we get that  for each  matrix function 
$\F- \A\in L^{r_1}(\Omega;\M^{N\times n}),$ 
with $r_1=\frac{q^*n}{n+2-q^*} $   
 holds 
$\bu\in  W_0^{1,r_1}\cap L^\infty(\Omega;\R^N)$ with the estimate
$$
\|D\bu\|_{r_1,\Omega}\leq C \|\F- \A\|_{r_1,\Omega}.
$$
Here  $ \A(x)=\{A_i^\alpha(x)\}_{i\leq N}^{\alpha\leq n}$ and $\F(x)=\{F_i^\alpha(x)\}_{i\leq N}^{\alpha\leq n}.$
Let us note that this  estimate is valid  for each solution of $\eqref{eq4}$ including $\bu.$
Repeating the above  procedure for  $\bu\in W^{1,r_1}(\Omega;\R^N)\cap L^\infty(\Omega;\R^N)$ 
we get that
$$
|D\bu|\in L^{r_2}(\Omega)\qquad   r_2=\min\Big\{p,q^\ast_\mu,  \big(\frac{r_1n}{n+2}\big)^* \Big\}.
$$
 If  
$ r_2=\min\{p, q^*_\mu \}$ then we have the assertion, otherwise $ r_2=(\frac{r_1n}{n+2})^*>r_1 $ and we repeat the arguments of the previous case. In such a way  we get an increasing sequence of indexes $\{r_k\}_{k\geq 0}.$  After 
$k'$ iterations we obtain $r_{k'}\geq \min\{p,q^*_\mu\} $  and 
\begin{equation}\label{r2}
\|D\bu\|_{r,\Omega}\leq C \|\F- \A\|_{r,\Omega}  \quad \text{ with } \quad r=\min \{p,q^*_\mu\}.
\end{equation}

The {\it second step} consists of showing that the gradient lies in a suitable  Morrey  space.  Suppose that $|D\bu|\in L^{r,\theta}(\Omega)$ with {\it  arbitrary} $\theta\in [0,n).$ Direct calculations give that $|D\bu|^{\frac{n+2}{n}}\in L^{\frac{rn}{n+2},\theta},$ i.e.   
$$
 \left(\frac{1}{\rho^\theta}\int_{\B_\rho} |D\bu|^{\frac{n+2}{n}\frac{rn}{n+2} } \,dx \right)^{\frac{n+2}{rn}}=
 \left(\frac{1}{\rho^\theta}\int_{\B_\rho} |D\bu|^r \,dx \right)^{\frac{n+2}{rn}}\leq 
\|D\bu\|_{r,\theta;\Omega}^{\frac{n+2}{n}}.
$$
 Keeping in mind \eqref{FAest} and \eqref{eqRiesz} we get 
$$
I_1 |D\bu|^{\frac{n+2}{n}}\in L^{(\frac{nr}{n+2})_\theta^*,\theta}(\Omega) 
$$
while $\varphi_1\in L^{p,\lambda}(\Omega)$ and $I_1\varphi_2\in L^{q^*_\mu,\mu}(\Omega).$ 

Further by the H\"older inequality we get the estimates 
\begin{align*}
&\left( \frac{1}{\rho^{n-\frac{n-\lambda}{p}r}  }\int_{\B_\rho}\varphi_1(x)^r\,dx \right)^\frac{1}{r}\leq C(n) \|\varphi_1\|_{p,\lambda;\Omega}\\
&\left( \frac{1}{\rho^{n-\frac{n-\mu}{q^*_\mu}r}  }\int_{\B_\rho}(I_1\varphi_2(x))^r\,dx \right)^\frac{1}{r}\leq
 C(n) \|I_1\varphi_2\|_{q^*_\mu,\mu;\Omega}
\end{align*}
that implies  $\varphi_1\in L^{r,n-\frac{n-\lambda}{p}r }(\Omega)$ and 
$I_1\varphi_2\in L^{r, n-\frac{n-\mu}{q^*_\mu}r}(\Omega).$

As it concerness the potential $I_1|D\bu|^{\frac{n+2}{n}}$ we consider two cases:
\begin{enumerate}
\item
$n-\theta\leq \frac{rn}{n+2}$ then $\big(\frac{nr}{n+2}\big)_\theta^*$ is arbitrary large number and we can take it such that 
$I_1|D\bu|^{\frac{n+2}{n}}\in L^r(\Omega);$
\item 
$n-\theta> \frac{rn}{n+2}$ then by the imbeddings   between  the  Morrey spaces  we have 
$$
L^{(\frac{nr}{n+2})_\theta^*,\theta}(\Omega)\subset L^{r, r-2 + \theta\frac{n+2}{n}}(\Omega)\,.
$$
Then 

$$
| F_i^\alpha-A_i^\alpha|\in L^{r,\min\{ r-2+\theta \frac{n+2}n, n-\frac{n-\lambda}{p}r, n-\frac{n-\mu}{q^*_\mu}r  \}}(\Omega)
$$
which implies via \cite[Theorem 5.1]{BSf} that the gradient of the solution of the  linearized problem satisfies
$$
|D\bu|\in    L^{r,\min\{ r-2+\theta \frac{n+2}n, n-\frac{n-\lambda}{p}r, n-\frac{n-\mu}{q^*_\mu}r  \}}(\Omega)\,.
$$
\end{enumerate}
In order to determine the optimal $\theta$ we use step-by-step arguments starting with the result obtained in the first step  and taking   as $\theta_0 =0.$ Suppose that
$$
r-2< \min\left\{  n-\frac{ n-\lambda }{p} r, n-\frac{n-\mu}{q^*_\mu}r  \right\},
$$ 
otherwise we have the assertion.
Repeating   the above procedure  with $\bu$ such that $|D\bu|\in L^{r,\theta_1}(\Omega)$ with $\theta_1=r-2$  we obtain
$$
|D\bu|\in    L^{r,\theta_2}(\Omega)
$$
 with 
$$ 
\theta_2=\min\left\{r-2+\theta_1\frac{n+2}n, n-\frac{n-\lambda}{p}r, n-\frac{n-\mu}{q^*_\mu}r  \right\}\,.
$$
If $\theta_2=\min\{ n-\frac{n-\lambda}{p}r, n-\frac{n-\mu}{q^*_\mu}r  \}$ we have the assertion,   otherwise we take $\theta_2= r-2+\theta_1\frac{n+2}n=(r-2)(1+\frac{n+2}{n}).$

 Iterating we obtain
an  increasing sequence
 $\{\theta_k =(r-2)\sum_{i=0}^{k-1}(\frac{n+2}{n})^i  \}_{k\geq 1}.$ Then there exists an index $k''$ for which 
$$
r-2+ \theta_{k''}\frac{n+2}{n}\geq \min\Big\{ n-\frac{n-\lambda}{p}r, n-\frac{n-\mu}{q^*_\mu}r \Big \}
$$
that gives the assertion. 

If  \fbox{$n=2$}  then the growth conditions have the form
\begin{align}\label{contr3}
 & |\ba(x,\bu) |\leq  \Lambda(\varphi_1(x)+|\bu|^{\varkappa}), \\
\nonumber
& \qquad\qquad   \varphi_1\in L^{p,\lambda}  (\Omega), \quad   p>2 , \      p+\lambda>n, \  \lambda\in[0,n),\\
\label{contr4}
&  |\bb(x,\bu,\bz)|\leq   \Lambda\big(\varphi_2(x) + |\bu|^{\varkappa -1}+|\bz|^{2-\epsilon}  \big), \\
\nonumber
&\qquad \qquad  \varphi_2\in L^{q,\mu} (\Omega), \quad   q>1, \    2q+\mu>n, \  \mu\in[0,n)
\end{align}
with $\varkappa>1$ arbitrary large number  and $\epsilon>0$ arbitrary small.  

Fixing again the solution $\bu\in W_0^{1,r_0}(\Omega;\R^N)\cup L^\infty (\Omega;\R^N)$  in the nonlinear terms and using the Lemma~\ref{lem1} and Lemma~\ref{lem2} we obtain
$$
F_i^\alpha-A_i^\alpha\in L^{r_1}(\Omega)\qquad r_1=\min\Big\{p,q^*_\mu,\Big( \frac{r_0}{2-\epsilon} \Big)^*   \Big\}\,.
$$
If $r_1=  \big( \frac{r_0}{2-\epsilon} \big)^* $ then the only possible value for $r_0$ is $r_0=q^*$ and hence $r_1=\frac{2q^*}{2(2-\epsilon)-q^*},$ otherwise we rich to contradiction.  Then by \cite{BW} we get 
$
| D\bu|\in L^{r_1}(\Omega).
$
Repeating the above procedure with $\bu\in   W_0^{1,r_1}\cap L^\infty(\Omega;\R^N)$ we obtain that
$$
|D\bu|\in L^{r_2}(\Omega)\qquad r_2=\min\Big\{p,q^*_\mu, \Big(  \frac{r_1}{2-\epsilon}\Big)^*   \Big\}\,.
$$
If 
$$
r_2=\Big( \frac{r_1}{2-\epsilon} \Big)^*<\min\{ p,q^*_\mu \}
$$
 we repeat the same  procedure  obtaining an increasing sequence $\{r_k\}_{k\geq 0}. $ Hence there exist an index $k_0$ such that $r_{k_0}\leq \min\{p,q^*_\mu\} $ that gives the assertion.

To obtain Morrey's regularity we take $|D\bu|\in L^{r,\theta}(\Omega)$  with {\it arbitrary} $\theta\in [0,2).$ Hence 
$|D\bu|^{2-\epsilon}\in L^{\frac{r}{2-\epsilon},\theta}(\Omega).$ By Lemma~\ref{lem1} and   Lemma~\ref{lem2}
 we obtain
\begin{align*}
\varphi_1&\in L^{p,\lambda}(\Omega)\subset L^{r,2-\frac{2-\lambda}{p}r}(\Omega)\\
I_1\varphi_2&\in L^{q^*_\mu,\mu}(\Omega)\subset L^{r,2-\frac{2-\mu}{q^*_\mu}r}(\Omega)\\
I_1|D\bu|^{2-\epsilon}&\in L^{(\frac{r}{2-\epsilon})_\theta^*,\theta}(\Omega)\subset L^{r, r-2(1-\epsilon)+\theta(2-\epsilon)}(\Omega).
\end{align*}
Hence  the Calder\'on-Zygmund estimate for the linearized problem (see \cite{BSf}) gives
$$
|D\bu|\in L^{r,\min\{2-\frac{2-\lambda}{p}r,2-\frac{2-\mu}{q^*_\mu}r,r-2(1-\epsilon)+\theta(2-\epsilon)\}}(\Omega)\,.
$$
To determine the precise Morrey space we applay the   step-by-step procedure. 
\begin{enumerate}
\item 
Since the last term is minimal when $\theta=0$ than we start with   an this initial value  $\theta_0=0.$ 
Suppose that 
$$
r-2(1-\epsilon)<\min\Big\{ 2-\frac{2-\lambda}{p}r, 2-\frac{2-\mu}{q^*_\mu}r  \Big\}<2
$$
(otherwise we have the assertion) and denote $\theta_1=r-2(1-\epsilon).$
\item Take $|D\bu|\in L^{r,\theta_1}(\Omega).$ The  above  procedure gives
$
|D\bu|\in L^{r,\theta_2} (\Omega)$    with 
$$
 \theta_2=\min\Big\{2-\frac{2-\lambda}{p}r, 2-\frac{2-\mu}{q^*_\mu}, 
r-2(1-\epsilon) +\theta_1(2-\epsilon)   \Big\}\,.
$$ 
If   $\theta_2=r-2(1-\epsilon) +\theta_1(2-\epsilon)$   (otherwise we have the assertion)
 then we continue with the same procedure obtaining the  sequence defined by recurrence
$$
\theta_0=0,\quad \theta_k=r-2(1-\epsilon) +\theta_{k-1}(2-\epsilon).
$$ 
\item
Since $r>2,$ hence the sequence is increasing  and there exists an index $\ol k$ such that 
$$
\theta_{\ol k}\geq \min\Big\{  2-\frac{2-\lambda}{p}r, 2-\frac{2-\mu}{q^*_\mu}r  \Big\}
$$
which is the assertion. 
\end{enumerate}
\end{proof}
\begin{crlr}
Supposing  the conditions of Theorem \ref{th1}, for any fixed $i=1,\ldots,N$ holds  $u^i\in C^{0,\alpha}(\Omega)$
with $\alpha=\min\big\{ 1-\frac{n-\lambda}{p},1-\frac{n-\mu}{q^*_\mu}   \big\}\,$ and for any ball $\B_\rho(z)\subset\Omega$
$$
\underset{\B_\rho(z)}{\rm osc} u^i\leq C \rho^\alpha\,.
$$
\end{crlr}
\begin{proof}
By \eqref{gradient} we have that for each ball $\B_\rho(z)\subset \Omega$
$$
\int_{\B_\rho(z)}|Du^i(y)|\,dy\leq  C \rho^{n-\frac{n-\nu}{r}}.
$$
Then for any $x,y\in \B_\rho(z)$ and for each fixed $i=1,\ldots,N$ we have
\begin{align*}
|u^i(x)-u^i(y)|&  \leq 2|u^i(x)-u^i_{\B_\rho(z)}|\leq C \int_{\B_\rho(z)} \frac{Du^i(y)}{|x-y|^{n-1}}\,dy\\
& \leq   C\int_0^\rho\int_{\B_t(z)}|D u^i(y)|\, dy\, \frac{dt}{t^n}\leq C\rho^{1-\frac{n-\nu}{r}}\,.
\end{align*}
\end{proof}

\subsection*{Acknowledgments.} 
The  both authors are members of  {\it   INDAM-GNAMPA }.  The authors are indebted of the referee for the valuable suggestions that improved the exposition of the paper.


\begin{thebibliography}{12}
\bibitem{AM}
Acerbi, E.,  Mingione, G., Gradient estimates for a class of parabolic systems, {\it  Duke Math. J.}  {\bf  136} (2007),   285-320.


\bibitem{Ad}
Adams, D.R.,  
A note on Riesz potentials, {\it Duke Math. J.}, {\bf 42} (4) (1975), 765--778.

\bibitem{Ar} 
Arkhipova, A.A., 
 Reverse H\"older inequalities with boundary integrals and $L_p$-estimates for
solutions of nonlinear elliptic and parabolic boundary-value problems, in {\it Nonl. Evol.
Equ.} (ed. N. N. Uraltseva), {\it Amer. Math. Soc. Transl.} Ser. 2, {\bf 164}, Amer. Math. Soc.,
Providence, RI, (1995), 15--42.

\bibitem{BP1}
Byun, S.-S.,  Palagachev, D.K.,
Morrey regularity of solutions to quasilinear elliptic equations over Reifenberg flat domains,
{\it   Calc. Var.}, {\bf 49} (2014),  37-–76. 


\bibitem{BPSh}
Byun, S.-S.,  Palagachev, D.  Shin, P.,
Sobolev–Morrey regularity of solutions to general quasilinear  elliptic equations,
{\it Nonlin. Anal., Theory Meth.  Appl., Ser. A, Theory Meth.},  {\bf 147} (2016), 176--190.

 

\bibitem{BSf} Byun, S.-S.,  Softova, L.,
Gradient estimates in generalized Morrey spaces for parabolic operators, 
{\it    Math. Nachr.,}  {\bf 288} (14-15) (2015), 1602--1614.


\bibitem{BS1} Byun, S.-S.,  Softova, L.
{ Asymptotically regular operators in generalized Morrey spaces},
{\it Bull. London Math. Soci.,} doi:10.1112/blms.12306 (in print).



\bibitem{BW1}
Byun, S.-S.,  Wang, L..
Elliptic equations with BMO coefficients in Reifenberg domains 
{\it   Commun. Pure Appl. Math.,}  {\bf  57} (10) (2004), 1283--1310.


\bibitem{BW2}
Byun, S.-S.,  Wang, L.,  Parabolic equations in time dependent Reifenberg domains,  {\it  Adv. Math.,} {\bf   212} (2) (2007), 
797--818.

\bibitem{BW}
 Byun, S.-S.,  Wang, L.,
{Gradient estimates for elliptic systems in non-smooth domains,} 
\textit{Math. Ann.,} \textbf{341} (2008), 629--650.



\bibitem{C} 
Campanato, S.,
Sistemi ellittici in forma divergenza. Regolarit\`a all'interno,
{\it  Pubbli. Classe  Scienze: Quaderni, Scuola Norm. Sup., Pisa, }  1980.




\bibitem{DeG}  De Giorgi, E.,
 Sulla differenziabilit\'a e l'analiticit\'a delle estremali degli integrali multipli regolary, 
{\it Mem. Accad. Sci. Torino Cl. Sci. Fis. Mat. Natur.,} {\bf 3} (3) (1957), 25--43.

\bibitem{DK}   Dong, H.,  Kim, D., 
 Global regularity of weak solutions to quasilinear elliptic and
parabolic equations with controlled growth, 
{\it Commun. Part. Differ. Equ.,} {\bf  36}  (2011), 1750--1777.

\bibitem{FMT}
Fattorusso, L.,  Molica Bisci, G.,  Tarsia, A.,
A global regularity result for some degenerate elliptic systems, 
{\it Nonlin. Anal., Theory Meth.  Appl., Ser. A,}, {\bf  125} (2015), 54--66.

\bibitem{FT}
Fattorusso, L.,  Tarsia, A.,
Morrey regularity of solutions of fully nonlinear elliptic systems, 
{\it Compl. Var. Elliptic Equ.,} {\bf  55} (5-6) (2010), 537--548.


\bibitem{Gh}  Gehring, F.W.,  $L_p$-integrability of the partial derivatives of a quasi conformal mapping, \textit{Acta Math.,} {\bf 130} (1973), 265--277. 


\bibitem{G} Giaquinta, M.,
\textit{Multiple Integrals in the Calculus of Variations and Nonlinear Elliptic Systems,} 
Annals of Mathematics Studies, 105, Princeton University Press, Princeton, NJ, 1983.


\bibitem{GMo} Giaquinta, M., Modica, G., 
Regularity results for some classes of higher order nonlinear elliptic systems,
 {\it J. Reine Angew. Math.,} {\bf 311/312} (1979), 437--451.

\bibitem{GT} Gilbarg, D., Trudinger, N.,
{\it Elliptic Partial Differential Equations of Second Order,}
Classics in Mathematics. Berlin, Springer  xiii, 517 p.,  2001.

\bibitem{GPVC}
Guarnieri, A.,  Pirotti, F.,  Vettore, A., Crocetto, N., 
 Digression on a particular Abel  integral, {\it  J. Interd. Math.,}
{\bf 6} (12) (2009), 863--874.   




\bibitem{JS} John, O., Stara, Y., On the regularity and nonregularity of elliptic and parabolic systems, Equadiff-7, Proc. Conf. Prague, 1989;  J. Kurzweil ed.,  Teubner-Texte Math. {\bf 118}, Teubner, Leipzig, 1990, 28--36. 


\bibitem{LU}
 Ladyzhenskaya, O.A.,  Ural'tseva, N.N.,
{\it Linear and Quasilinear Equations of Elliptic Type,} 2nd Edition, Nauka, Moscow,  1973, (in Russian).


\bibitem{LP}     Leonetti, F.,  Petricca, P.V., 
Regularity for solutions to some nonlinear elliptic systems,
{\it Compl. Var. Elliptic Eq.} (12) {\bf 56}  (2011), 1099-1113.


\bibitem{MM} Marino, M., Maugeri, A., 
Boundary regularity results for non-variational basic elliptic systems, {\it  Le Matematiche,}  {\bf  55} (2) (2000) 109--123. 




\bibitem{Mo}    Morrey, C.B.,
  On the solutions of quasi-linear elliptic partial
  differential equations, {\em  Trans.  Amer.   Math.  Soci.,}  {\bf 43} (1938), 126--166.

\bibitem{Morrey1} Morrey, C.B.,
Second order elliptic equations in several variables and H\"older continuity,
 {\it Mathem. Zeitschr.,} {\bf 72} (1959), 146--164.



\bibitem{PS1} Palagachev, D.K., Softova, L.G., 
Divergence form parabolic equations in Reifenberg flat domains, {\it Discr. Cont. Dynam. Syst.,}  {\bf 31} (4)   (2011), 1397--1410.


\bibitem{PS2} Palagachev, D.K., Softova, L.G., 
The Calder\'on-Zygmund property for quasilinear divergence form equations over Reifenberg flat domains, 
{\it Nonlin. Anal., Theory Meth.  Appl., Ser. A},  {\bf 74} (2011),  1721--1730.

\bibitem{PSf3} Palagachev, D.K., Softova, L.G., 
Boundedness of solutions to a class of coercive systems with Morrey data,  {\it Nonlin. Anal., Theory Meth.  Appl., Ser. A}, {\bf 191} (2020), 111630, 16 pp.




\bibitem{Pi} Piccinini, L.C.,
 Inclusioni tra spazi di Morrey,
{\it Boll. Unione Mat. Ital.,} Ser IV (2) (1969), 95--99.

\bibitem{Sf2} Softova, L.,
$L^p$ -integrability of the gradient of solutions to quasilinear systems with discontinuous coefficients, {\it 
Differ. Integral Equ.,} {\bf  26}  (9-10) (2013), 1091--1104. 

\bibitem{Sf3} Softova, L., 
 Maximum principle for a kind of elliptic systems with Morrey data,
{\it Different. Difference Equ. Appl.},  Springer Proc. Math. Stat., {\bf  230}, Springer, 2018,  429--439.

\bibitem{Sf1} Softova, L.G.,
Boundedness of the  solutions to  nonlinear  systems with  Morrey data,  
{\it Compl. Var. Elliptic  Equ.,}
 {\bf 63}  (11)  (2018).

\bibitem{R}
 Reifenberg, E.R., 
{Solution of the Plateau problem for $m$-dimensional surfaces of varying topological type,}
\textit{Acta Math.} \textbf{104} (1960), 1--92.

\bibitem{T}
 Toro, T.,
{Doubling and flatness: geometry of measures,}
\textit{Notices Amer. Math. Soc.} \textbf{44} (1997), 1087--1094.
\end{thebibliography}
\end{document}